\documentclass{amsart}
\usepackage{amssymb,stmaryrd}

\newtheorem{theorem}{Theorem}[section]
\newtheorem{claim}{Claim}[theorem]
\newtheorem{lemma}[theorem]{Lemma}
\newtheorem{fact}{Fact}	
\newtheorem{prop}[theorem]{Proposition}
\newtheorem{cor}[theorem]{Corollary}

\newtheorem*{thma}{Theorem~A}
\newtheorem*{thmb}{Theorem~B}
\newtheorem*{thmc}{Theorem~C}
\newtheorem*{cora}{Corollary~1}
\newtheorem*{corb}{Corollary~2}

\theoremstyle{definition}
\newtheorem{defn}[theorem]{Definition}

\theoremstyle{remark}
\newtheorem{q}[theorem]{Question}

\DeclareMathOperator{\add}{Add}
\DeclareMathOperator{\otp}{otp}
\DeclareMathOperator{\acc}{acc}
\DeclareMathOperator{\nacc}{nacc}
\DeclareMathOperator{\ads}{ADS}
\DeclareMathOperator{\cf}{cf}
\DeclareMathOperator{\cov}{cov}
\DeclareMathOperator{\dom}{dom}
\DeclareMathOperator{\im}{Im}
\DeclareMathOperator{\ns}{NS}
\newcommand*\axiomfont[1]{\textsf{\textup{#1}}}
\newcommand\ssh{\axiomfont{SSH}}
\renewcommand\mid{\mathrel{|}\allowbreak}
\def\s{\subseteq}

\subjclass[2010]{Primary 03E05; Secondary 03E35, 03E17}

\title{On the ideal $J[\kappa]$}
\author{Assaf Rinot}
\address{Department of Mathematics, Bar-Ilan University, Ramat-Gan 5290002, Israel.}
\urladdr{http://www.assafrinot.com}

\begin{document}
\begin{abstract}  Motivated by a question from a recent paper by Gilton, Levine and Stejskalov{\'a},
we obtain a new characterization of the ideal $J[\kappa]$, from which we confirm that $\kappa$-Souslin trees 
exist in various models of interest.

As a corollary we get that for every integer $n$ such that $\mathfrak b<2^{\aleph_n}=\aleph_{n+1}$,
if $\square(\aleph_{n+1})$ holds,
then there exists an $\aleph_{n+1}$-Souslin tree.
\end{abstract}
\maketitle

\section{Introduction}
In a recent paper, Gilton, Levine and Stejskalov{\'a} prove two consistency results from the consistency of a Mahlo cardinal.
Focusing on the first interesting case, being the cardinal $\aleph_2$, these results read as follows.
\begin{fact}[{\cite[Theorem~1.4 and Section~6.3]{gilton2021trees}}]
If $\mu$ is a Mahlo cardinal, then in some forcing extension, $\mu=\aleph_2$, and all of the following hold simultaneously:
\begin{enumerate}
\item $S^2_1\in I[\aleph_2]$;
\item $2^{\aleph_0}=2^{\aleph_1}=\aleph_2$;
\item Every stationary subset of $S^2_0$ reflects;
\item There is an $\aleph_2$-Aronszajn tree, but all $\aleph_2$-Aronszajn trees are nonspecial.
\end{enumerate}

Furthermore, Clauses (1),(3),(4) are indestructible under the forcing to add any number of Cohen reals.
\end{fact}

\begin{fact}[{\cite[Theorem~1.5 and Section~6.3]{gilton2021trees}}]
If $\mu$ is a Mahlo cardinal, then in some forcing extension, $\mu=\aleph_2$, and all of the following hold simultaneously:
\begin{itemize}
\item[(1')] $S^2_1\notin I[\aleph_2]$;
\item[(2)] $2^{\aleph_0}=2^{\aleph_1}=\aleph_2$;
\item[(3)] Every stationary subset of $S^2_0$ reflects;
\item[(4)] There is an $\aleph_2$-Aronszajn tree, but all $\aleph_2$-Aronszajn trees are nonspecial.
\end{itemize}

Furthermore, Clauses (1'),(3),(4) are indestructible under the forcing to add any number of Cohen reals.
\end{fact}

At the end of their paper they ask whether in each of these models there moreover exists an $\aleph_2$-Souslin tree.
As a first step, and even without inspecting their particular models, one can utilize the powerful indestructibility feature of these models
to prove that (1),(2),(3),(4) and (1'),(2),(3),(4) are both consistent with the existence of an $\aleph_2$-Souslin tree:

\begin{cora} Assuming the consistency of a Mahlo cardinal, it is consistent that (1)--(5) hold simultaneously,
and it is consistent that (1') and (2)-(5) hold simultaneously, where:
\begin{itemize}
\item[(5)] There exists a uniformly coherent $\aleph_2$-Souslin tree.
\end{itemize}
\end{cora}

This will follow from the fact that adding  $\aleph_1$ many Cohen reals makes $\mathfrak b=\aleph_1$ together with the following  theorem.
\begin{thma} Suppose that $n<\omega$ is such that $\mathfrak b<2^{\aleph_n}=\aleph_{n+1}$ and $\square(\aleph_{n+1})$ holds.
Then there is an $\aleph_{n+1}$-Souslin tree.
If every stationary subset of $S^{n+1}_0$ reflects,
then there is even a uniformly coherent $\aleph_{n+1}$-Souslin tree.
\end{thma}

To prove Theorem~A, we shall deepen our study of the ideal $J[\kappa]$ from \cite{paper24},
and its enrichment $J_\omega[\kappa]$ which we introduce here.
The two are $\kappa$-complete normal ideals over a regular uncountable cardinal $\kappa$,
and if $\square(\kappa)$ and $\diamondsuit(\kappa)$ both hold and  $J_\omega[\kappa]\neq\ns_\kappa$,
then there exists a $\kappa$-Souslin tree.
Here, we shall find sufficient and equivalent conditions for the ideals to contain a stationary set.
At the level of $\aleph_2$, this reads as follows.
\begin{thmb}\begin{itemize}
\item $J_\omega[\aleph_2]\neq\ns_{\aleph_2}$
iff $\mathfrak b=\aleph_1$;
\item  $J[\aleph_2]\neq\ns_{\aleph_2}$
iff there exists a nonmeager set of reals of size $\aleph_1$.
\end{itemize}
\end{thmb}

Thus, coming back to the question of whether an $\aleph_2$-Souslin tree exists in the original models from \cite{gilton2021trees},
it turns out that the answer is affirmative provided that the Mahlo cardinal $\mu$ taken there is not weakly compact in $L$.
This anti-large-cardinal hypothesis on $\mu$ suffices to secure that $\square(\aleph_2)$ will hold in their models.
In addition, the models of Facts 1 and 2 are generic extensions by $\add(\omega,\mu)$ followed by an $\aleph_1$-distributive
forcing, meaning that the final models do satisfy $\mathfrak b=\aleph_1$.

\begin{corb} Assuming that $\mu$ is a Mahlo cardinal which is not weakly compact in $L$,
in the forcing extensions of Facts 1 and 2, there exists a uniformly coherent $\aleph_2$-Souslin tree.
\end{corb}

Let us make clear that the results of this paper are not limited to $\aleph_2$,
and more general results are given by Corollary~\ref{cor310} below.
This general study also yields (Corollary~\ref{cor33} below) that assuming Shelah's Strong Hypothesis,
$J[\lambda^+]\neq\ns_{\lambda^+}$ for every cardinal $\lambda>2^{\aleph_1}$, 
as well as uncovers the following connection between the ideal $J[\lambda^+]$ and Shelah's approachability ideal $I[\lambda^+]$:

\begin{thmc} For infinite cardinals $\theta<\lambda$, 
if $J[\lambda^+]$ contains a stationary subset of $E^{\lambda^+}_\theta$, then it contains every stationary subset of $E^{\lambda^+}_\theta$ that lies in $I[\lambda^+]$.

In particular, for $\lambda$ regular, if $J[\lambda^+]$ contains a stationary subset of $E^{\lambda^+}_\theta$, then $E^{\lambda^+}_\theta$ itself is in $J[\lambda^+]$.
\end{thmc}

\subsection{Notation and preliminaries} 
The definition of $\kappa$-Aronszajn trees and (uniformly coherent) $\kappa$-Souslin trees may be found in \cite[\S2]{paper22}.
We write $[\Delta]^\theta$ (resp.~$[\Delta]^{<\theta}$) for the collection of all subsets of $\Delta$ of cardinality (resp.~less than) $\theta$.
For a set $x$ of ordinals, we let $\acc(x) := \{\alpha\in x \mid \sup(x \cap \alpha) = \alpha > 0\}$,
and $\nacc(x):=x\setminus\acc(x)$.
A stationary subset $T$ of a regular uncountable cardinal $\kappa$ is said to \emph{reflect}
iff there exists an ordinal $\alpha<\kappa$ of uncountable cofinality such that $T\cap\alpha$ is stationary in $\alpha$.
$E^\kappa_\theta$ stands for $\{\alpha<\kappa\mid \cf(\alpha)=\theta\}$,
and $S^m_n$ stands for $E^{\aleph_m}_{\aleph_n}$.
$\ns_\kappa$ stands for the nonstationary ideal over $\kappa$,
$I[\kappa]$ stands for Shelah's approachability ideal (see \cite[Definition~3.3]{eisworth}),
and the definition of the ideal $J[\kappa]$ is reproduced in Section~\ref{theideal}.

\newpage
\section{Meetings numbers}
Throughout this section, 
$\lambda$ denotes an infinite cardinal,
$\theta$ denotes an infinite regular cardinal,
and $\chi$ is either $2$ or an infinite cardinal.

\begin{defn}[meeting numbers] Let $\Delta$ be some set of size $\ge\theta$.
\begin{itemize}
\item \cite[Definition~2.7]{Matet2021} $m(\theta,\lambda)$ stands for the least size of a subfamily $\mathcal Y\s[\lambda]^\theta$ 
such that, for every $X\in[\lambda]^\theta$, there exists $Y\in\mathcal Y$ with $|X\cap Y|=\theta$;
\item $m(\Delta,\lambda,\theta)$ stands for $m(\Delta,\lambda,\theta,2)$, where:
\item $m(\Delta,\lambda,\theta,\chi)$ stands for the least size of a family $\mathcal H$ of functions from $\Delta$ to $[\lambda]^{<\chi}$
such that, for every $X\in[\Delta]^\theta$ and every function $g:X\rightarrow\lambda$,
there exists $h\in\mathcal H$ with $|\{ \xi\in X\mid g(\xi)\in h(\xi)\}|=\theta$.
\end{itemize}
\end{defn}

\begin{prop}\label{prop22} \begin{enumerate}
\item If $\theta=\cf(\lambda)<\lambda$, then $m(\theta,\lambda)>\lambda$;
\item If $\theta<\lambda$, then $m(\theta,\lambda)\ge\lambda$;
\item If $\theta<\lambda<\theta^{+\theta}$, then $m(\theta,\lambda)=\lambda$.
\end{enumerate}
\end{prop}
\begin{proof} (1) By a diagonalization argument similar to the proof of K\"onig's theorem.

(2) This is clear.

(3) By induction on $\lambda$. For the base case $\lambda:=\theta^+$,
the set $\mathcal Y_\lambda:=\{ \theta+\alpha\mid \alpha<\lambda\}$ witnesses that $m(\theta,\lambda)=\lambda$.
Next, given a cardinal $\lambda$ 
and a family $\mathcal Y_\lambda$ witnessing that $m(\theta,\lambda)=\lambda$,
we fix a bijection $e_\alpha:\lambda\leftrightarrow\lambda+\alpha$ for every $\alpha<\lambda^+$, 
and then set $\mathcal Y_{\lambda^+}:=\{ e_\alpha[Y]\mid Y\in\mathcal Y, \alpha<\lambda^+\}$.
Clearly, $\mathcal Y_{\lambda^+}$ witnesses that $m(\theta,\lambda^+)=\lambda^+$.

Finally, for every limit cardinal $\lambda$ with $\theta<\lambda<\theta^{+\theta}$
such that for every cardinal $\mu\in(\theta,\lambda)$, there is a family $\mathcal Y_\mu$ witnessing that $m(\theta,\mu)=\mu$,
we let $\mathcal Y_\lambda:=\bigcup\{\mathcal Y_\mu\mid \theta<\mu=|\mu|<\lambda\}$,
so that $|\mathcal Y_\lambda|=\lambda$.
As $\lambda$ is a limit cardinal in  $(\theta,\theta^{+\theta})$, $\cf(\lambda)<\theta$.
In effect, for every set $X\in[\lambda]^\theta$, there exists some cardinal $\mu$ with $\theta<\mu<\lambda$ such that $X\cap\mu\in[\mu]^\theta$,
and then we can find $Y\in\mathcal Y_\mu\s\mathcal Y_\lambda$ such that $|X\cap Y|=\theta$.
\end{proof}

To help the reader digest the $3$-cardinal and $4$-cardinal variants, we offer the next proposition.

\begin{prop}\label{bprop} $m(\aleph_0,\aleph_0,\aleph_0)$ is equal to the least size of a nonmeager set of reals,
and $m(\theta,\theta,\theta,\theta)=\mathfrak b_\theta$.
\end{prop}
\begin{proof} The first part is given by \cite[Lemma~2.4.8(2)]{MR1350295}.
We shall next prove that 
 $m(\theta,\theta,\theta,\theta)\le \mathfrak b_\theta$,
 and then prove that $m(\theta,\theta,\theta,\theta)\ge \mathfrak b_\theta$.

$(\le)$: Let $\mathcal H\s{}^\theta\theta$ be an unbounded family in $({}^\theta\theta,{<^*})$.
Without loss of generality, we may assume that $\mathcal H$ consists of strictly increasing functions.
Each $h\in\mathcal H$ is a function from $\theta$ to $\theta$, but since every element of $\theta$ may be understood as an element of $[\theta]^{<\theta}$,
the function $h$ may be understood as a function from $\theta$ to $[\theta]^{<\theta}$.
Now, suppose that we are given a function $g:X\rightarrow\theta$ with $X\in[\theta]^\theta$.
Let $\pi:\theta\leftrightarrow X$ denote the increasing enumeration of $X$.
Set $f:=g\circ \pi$.
As $\mathcal H$ is unbounded, let us pick some $h\in\mathcal H$ such that $T:=\{ \tau<\theta\mid f(\tau)<h(\tau)\}$ has size $\theta$.
Now, $X':=\pi[T]$ is an element of $[X]^\theta$, and for every $\xi\in X'$,
if we let $\tau:=\pi^{-1}(\xi)$,
then $\tau\le \xi$ and $g(\xi)=f(\tau)<h(\tau)\le h(\xi)$,
so that $g(\xi)\in h(\xi)$, as sought.

\medskip

$(\ge)$: Let $\mathcal H$ be a family of functions from $\theta$ to $[\theta]^{<\theta}$
with the property that for every $X\in[\theta]^\theta$ and every function $g:X\rightarrow\theta$,
there exists $h\in\mathcal H$ for which $\{ \xi\in X\mid g(\xi)\in h(\xi)\}$ has size $\theta$.
For each $h\in\mathcal H$, define $h':\theta\rightarrow\theta$ via $h'(\xi):=\sup(h(\xi))+1$.
To see that $\{ h'\mid h\in\mathcal H\}$ is unbounded in $({}^\theta\theta,{<^*})$, let $g:\theta\rightarrow\theta$ be arbitrary.
Pick $h\in\mathcal H$ such that $\{ \xi<\theta\mid g(\xi)\in h(\xi)\}$ has size $\theta$.
Then $\neg(h'<^* g)$.
\end{proof}

\begin{lemma}\label{lemma23b} Suppose that $\theta<\lambda$. Then:
\begin{enumerate}
\item $\lambda\le m(\theta,\lambda)\le \min\{\cf([\lambda]^\theta,{\s}),\cf([\lambda]^\theta,{\supseteq})\}$;
\item if $\chi\le\theta^+$, then $m(\theta,\lambda)\le m(\theta,\lambda,\theta,\chi)$;
\item $m(\theta,\theta,\theta,\chi)\le m(\lambda,\lambda,\theta,\chi)\le\max\{m(\theta,\theta,\theta,\chi),m(\theta,\lambda)\}$;
\end{enumerate}
\end{lemma}
\begin{proof} (1) This is clear.

(2) Suppose that $\chi\le\theta^+$,
and let $\mathcal H$ be a family witnessing the value of $m(\theta,\lambda,\theta,\chi)$. Set $\mathcal Y:=\{\bigcup\im(h)\mid h\in\mathcal H\}\cap[\lambda]^\theta$.
Now, given $X\in[\lambda]^\theta$, if we let $g:\theta\leftrightarrow X$ be some bijection, then we may find $h\in\mathcal H$
with $|\{ \xi\in X\mid g(\xi)\in h(\xi)\}|=\theta$. So $Y:=\bigcup\im(h)$ satisfies $|Y\cap X|=\theta$.
In particular, $|Y|\ge\theta$. As $\chi\le\theta^+>|\dom(h)|$, it altogether follows that $|Y|=\theta$, so that $Y\in \mathcal Y$.

(3) 
We focus on the second inequality.
Let $\mathcal H$ be a family witnessing the value of $m(\theta,\theta,\theta,\chi)$,
and let $\mathcal Y\s[\lambda]^\theta$ be a family witnessing the value of $m(\theta,\lambda)$.
For each $Y\in\mathcal Y$, fix a bijection $f_Y:\theta\leftrightarrow Y$.
For all $h:\theta\rightarrow[\theta]^{<\chi}$ in $\mathcal H$ and $Y,Z\in\mathcal Y$,
define $h^Y_Z:\lambda\rightarrow[Z]^{<\chi}$ via:
$$h^Y_Z(\xi):=\begin{cases}
\emptyset,&\text{if }\xi\notin Y;\\
f_Z[h(f_{Y}^{-1}(\xi))]&\text{otherwise}.
\end{cases}$$
For every $\gamma<\lambda$, let $c_\gamma:\lambda\rightarrow[\lambda]^1$
denote the unique function to satisfy $c_\gamma(\xi)=\{\gamma\}$ for all $\xi<\lambda$.
Set $\mathcal H^*:=\{ h^Y_Z,c_\gamma\mid h\in\mathcal H,\  Y,Z\in\mathcal Y,\ \gamma<\lambda\}$,
so that $|\mathcal H^*|=\max\{|\mathcal H|,|\mathcal Y|\}$.
\begin{claim} Let $g:X\rightarrow\lambda$ with $X\in[\lambda]^\theta$.
Then there exists $h\in\mathcal H^*$ with $|\{ \xi\in X\mid g(\xi)\in h(\xi)\}|=\theta$.
\end{claim}
\begin{proof} 
If there exists $X'\in[X]^\theta$ such that $g\restriction X'$ is constant,
then there exists some $\gamma<\lambda$ such that $|\{ \xi\in X\mid g(\xi)\in c_\gamma(\xi)\}|=\theta$, and we are done.

Assume now that $g$ is ${<}\theta$-to-one. 
Since $\theta$ is regular and by passing to a subset of $X$ of size $\theta$, we may assume that $g$ is injective.
Pick $Y\in\mathcal Y$ such that $X\cap Y$ has size $\theta$.
As $g$ is injective, $g[X\cap Y]$ has size $\theta$, so we may pick $Z\in\mathcal Y$ such that $g[X\cap Y]\cap Z$ has size $\theta$.

By the choice of $Z$, $T:=\{\tau<\theta\mid f_Y(\tau)\in X\ \&\ g(f_Y(\tau))\in Z\}$ is in $[\theta]^\theta$.
Set $g':=(f_Z^{-1}\circ g\circ f_Y)\restriction T$, so that $g'$ is a function from $T$ to $\theta$.
Next, pick $h\in\mathcal H$ such that $T^*:=\{ \tau\in T\mid g'(\tau)\in h(\tau)\}$ has size $\theta$.
In effect, $X^*:=f_Y[T^*]$ is a subset of $X$ of size $\theta$.
Let $\xi\in X^*$. Then $\xi\in Y$, $f_Y^{-1}(\xi)\in T^*\s T$ and
$$g(\xi)=g(f_Y(f_Y^{-1}(\xi)))=f_Z(g'(f_Y^{-1}(\xi)))\in f_Z[h(f_Y^{-1}(\xi))]=h^Y_Z(\xi),$$
as sought.
\end{proof}
This completes the proof.
\end{proof}

\section{The ideal}\label{theideal}

Throughout this section, 
$\kappa$ denotes a regular uncountable cardinal,
$\lambda$ denotes an infinite cardinal,
$\theta$ denotes an infinite regular cardinal,
and $\chi$ is either $2$ or an infinite cardinal.

\begin{defn}\label{regressive}  $J_\chi[\kappa]$
stands for the collection of all subsets $S\s\kappa$ for which there exist a club $C\s\kappa$ and a sequence of functions  $\langle f_i:\kappa\rightarrow[\kappa]^{<\chi}\mid i<\kappa\rangle$ 
with the property that for every $\alpha\in S\cap C$, every regressive function $f:\alpha\rightarrow\alpha$,
and every cofinal subset $B\s\alpha$, there exists $i<\alpha$ such that $\sup\{ \beta\in B\mid f(\beta)\in f_i(\beta)\}=\alpha$.
\end{defn}

The ideal $J[\kappa]$ of \cite[Definition~2.1]{paper24} is nothing but $J_2[\kappa]$.
The two objects are quite close, 
as $\{ S\in J_\chi[\kappa]\mid S\s E^\kappa_{>\chi}\}\s J_2[\kappa]\s J_\chi[\kappa]$.
Also, the same proof of \cite[Proposition~2.14]{paper24} yields the following proposition.
\begin{prop}\label{inaccessible} If $\chi<\kappa$ and $\kappa$ is inaccessible,
then $J_\chi[\kappa]=\ns_\kappa$.\qed
\end{prop}
Thus, hereafter we shall study $J_\chi[\kappa]$ when $\kappa$ is of the form $\lambda^+$.
The next proposition is an easy improvement of \cite[Proposition~2.8]{paper24}.
\begin{prop}\label{prop35} Suppose that $\chi<\lambda$.
For every $S\in J_\chi[\lambda^+]$, $S\cap E^{\lambda^+}_{\cf(\lambda)}$ is nonstationary.
\end{prop}
\begin{proof} Suppose not. Then, in particular, we may fix $\alpha\in E^{\lambda^+}_{\cf(\lambda)}\setminus(\lambda+1)$ and a sequence of functions 
 $\langle f_i:\alpha\rightarrow[\lambda]^{<\chi}\mid i<\alpha\rangle$ satisfying the following.
For every regressive function $f:\alpha\rightarrow\lambda$,
and every cofinal subset $B\s\alpha$, there exists some $i<\alpha$ such that $\sup\{ \beta\in B\mid f(\beta)\in f_i(\beta)\}=\alpha$.
Let $\langle \beta_j \mid j<\cf(\lambda)\rangle$ be a strictly increasing sequence of ordinals, converging to $\alpha$, with $\beta_0>\lambda$.
Fix a decomposition $\alpha=\biguplus_{\eta<\cf(\lambda)}a_\eta$ such that $|a_\eta|<\lambda$ for all $\eta<\cf(\lambda)$.
Evidently, for every $j<\cf(\lambda)$, $|\bigcup_{\eta<j}\bigcup_{i\in a_\eta}\bigcup f_i(\beta_j)|<\lambda$.
Thus, we may pick a regressive function $f:\alpha\rightarrow\lambda$ satisfying that, for all $j<\cf(\lambda)$,
\begin{equation}\label{eq1}f(\beta_j)=\min\left(\lambda\setminus \bigcup_{\eta<j}\bigcup_{i\in a_\eta}\bigcup f_i(\beta_j)\right).\end{equation}
Set $B:=\{ \beta_j\mid j<\cf(\lambda)\}$ and fix $i<\alpha$ such that $\sup\{ \beta\in B\mid f(\beta)\in f_i(\beta)\}=\alpha$.
In particular, $J:=\{ j<\cf(\lambda)\mid f(\beta_j)\in f_i(\beta_j)\}$ is cofinal in $\cf(\lambda)$.
Find $\eta<\cf(\lambda)$ such that $i\in a_\eta$.
For every $j<\cf(\lambda)$ above $\eta$,  Equation~\eqref{eq1} entails that $f(\beta_j)\neq f_i(\beta_j)$.
This contradicts the fact that $J$ contains elements above $\eta$.
\end{proof}

For the sake of this paper, we introduce the following ad-hoc definition.

\begin{defn} For a map $d:\theta\times\kappa\rightarrow\lambda$, 
let $A(d)$ denote the set of all $\alpha<\kappa$ such that, for every cofinal subset $B\s\alpha$,
there exist $\eta\in\theta\cap\alpha$ and a cofinal $B'\s B$ on which $\beta\mapsto d(\eta,\beta)$ is injective.
\end{defn}

Note that $A(d)$ contains all successor ordinals, but we shall only care about the nonzero limit members of $A(d)$.

\begin{lemma}\label{lemma23} Suppose that $m(\lambda,\lambda,\theta,\chi)=\lambda$.
Then, for every map $d:\lambda^+\times\lambda^+\rightarrow\lambda$,
$A(d)\cap E^{\lambda^+}_\theta\in J_\chi[\lambda^+]$.
\end{lemma}
\begin{proof} 
For every $\beta<\lambda^+$, fix a surjection $e_\beta:\lambda\rightarrow\beta+1$.
Fix a bijection $\pi:\lambda^+\leftrightarrow\lambda\times\lambda^+$,
and consider the club $C:=\{ \alpha<\lambda^+\mid \pi[\alpha]=\lambda\times\alpha\}$.
Fix a sequence  $\langle h_j \mid j<\lambda\rangle$ of functions from $\lambda$ to $[\lambda]^{<\chi}$ such that, 
for every $X\in[\lambda]^\theta$ and every function $g:X\rightarrow\lambda$,
there exists $j<\lambda$ with $|\{\xi\in X\mid g(\xi)\in h_j(\xi)\}|=\theta$.
Now, given a map $d:\lambda^+\times\lambda^+\rightarrow\lambda$,
set $S:=A(d)\cap E^{\lambda^+}_\theta$.
Finally, for each $i<\lambda^+$, 
derive a function $f_i:\lambda^+\rightarrow[\lambda^+]^{<\chi}$ via 
$$f_i(\beta):=e_\beta[h_j(d(\eta,\beta))],$$
where $(j,\eta):=\pi(i)$.

\begin{claim} Let $\alpha\in S\cap C$. Suppose that $f:\alpha\rightarrow\alpha$
is a regressive function, and $B\s\alpha$ is cofinal.
Then there exists $i<\alpha$ such that $\sup\{ \beta\in B\mid f(\beta)\in f_i(\beta)\}=\alpha$.
\end{claim}
\begin{proof}
By passing to a cofinal subset, we may assume that $\otp(B)=\theta$. 
As $f$ is regressive, we may pick a function $f^*:\alpha\rightarrow\lambda$ such that $e_\beta(f^*(\beta))=f(\beta)$ for every $\beta<\alpha$.
Find $\eta<\alpha$ and a cofinal subset $B'\s B$ on which $\beta\mapsto d(\eta,\beta)$ is injective.
Set $X:=\{ d(\eta,\beta)\mid \beta\in B'\}$, so that $X\in[\lambda]^\theta$.
Define a function $g:X\rightarrow\lambda$ via $g(d(\eta,\beta)):=f^*(\beta)$.
Now, pick $j<\lambda$ such that $X^*:=\{\xi\in X\mid g(\xi)\in h_j(\xi)\}$ has size $\theta$.
Evidently $B^*:=\{ \beta\in B'\mid d(\eta,\beta)\in X^*\}$ has size $\theta$,
so that it is cofinal in $\alpha$. Finally, as $\eta\in\alpha\in C$, we may find $i<\alpha$ such that $\pi(i)=(j,\eta)$. 
Then, for every $\beta\in B^*$,
\[f(\beta)=e_\beta(f^*(\beta))=e_\beta(g(d(\eta,\beta)))\in e_\beta[h_j(d(\eta,\beta))]=f_i(\beta).\qedhere\]
\end{proof}
This completes the proof.
\end{proof}

\begin{prop}\label{sufficient}
\begin{enumerate}
\item There is a map $d:\cf(\lambda)\times\lambda^+\rightarrow\lambda$ such that $E^{\lambda^+}_{<\cf(\lambda)}\s A(d)$;
\item If $\ads_\lambda$ holds, then there is a map $d:\cf(\lambda)\times\lambda^+\rightarrow\lambda$ such that $E^{\lambda^+}_{\neq\cf(\lambda)}\s A(d)$;
\item If $\lambda$ is singular and $\vec f=\langle f_\beta\mid \beta<\lambda^+\rangle$ is a scale for $\lambda$,
then there is a map $d:\cf(\lambda)\times\lambda^+\rightarrow\lambda$ such that $G(\vec f)\s A(d)$;
\item If $(\aleph_{\omega+1}, \aleph_\omega) \twoheadrightarrow (\aleph_1, \aleph_0)$ holds,
then for every map $d:\aleph_0\times\aleph_{\omega+1}\rightarrow\aleph_\omega$,  $E^{\aleph_{\omega+1}}_{\aleph_1}\setminus A(d)$ is stationary.
\end{enumerate}
\end{prop}
\begin{proof}  (1) For $\lambda$ regular, this follows from Clause~(2), and for $\lambda$ singular,
this follows from Clause~(3). 

(2) Recall that $\ads_\lambda$ is the principle from \cite[p.440]{MR675955} asserting the existence of a sequence $\vec A=\langle A_\beta\mid \beta<\lambda^+\rangle$ of cofinal subsets of $\lambda$,
with the property that, for every $\alpha<\lambda^+$, there exists a function $f_\alpha:\alpha\rightarrow\lambda$ such that $\langle A_\beta\setminus f_\alpha(\beta)\mid \beta<\alpha\rangle$
consists of pairwise disjoint sets. It is easy to see that if $\lambda$ is regular, then $\ads_\lambda$ holds.
By \cite[Theorem~4.1]{MR1838355}, if there exists a special $\lambda^+$-Aronszajn tree, then $\ads_\lambda$ holds.

Fix a strictly increasing sequence of ordinals $\langle \lambda_\eta\mid \eta<\cf(\lambda)\rangle$
converging to $\lambda$ (if $\lambda$ is regular, then we may just let $\lambda_\eta:=\eta$).
Define $d:\cf(\lambda)\times\lambda^+\rightarrow\lambda$ via:
$$d(\eta,\beta):=\min(A_\beta\setminus\lambda_\eta).$$

Now, given a subset $B\s\lambda$ for which $\alpha:=\sup(B)$ in $E^{\lambda^+}_{\neq\cf(\lambda)}$,
it is easy to find $\eta<\cf(\lambda)$ and a cofinal $B'\s B$ such that $f_{\alpha}(\beta)\le\lambda_\eta$ for all $\beta\in B'$ (see \cite[Fact~1.3]{paper7}).
In effect, $\beta\mapsto d(\eta,\beta)$ is injective over $B'$.

(3) Recall that $G(\vec f)$ stands for the set of all $\alpha\in E^{\lambda^+}_{\neq\cf(\lambda)}$ 
for which there exist a cofinal subset $A\s\alpha$ and $i<\cf(\lambda)$ such that, for every pair of ordinals $\gamma<\delta$ from $A$ and every $j\in[i,\cf(\lambda))$, $f_\gamma(j)<f_{\delta}(j)$.
Evidently, $E^{\lambda^+}_{<\cf(\lambda)}\s G(\vec f)$.

Define $d:\cf(\lambda)\times\lambda^+\rightarrow\lambda$ via $d(\eta,\beta):=f_\beta(\eta)$.
To see this works, let $\alpha\in G(\vec f)$. To avoid trivialities, suppose that $\theta:=\cf(\alpha)$ is infinite.
Let $B$ be a cofinal subset of $\alpha$.
Fix $A\s\alpha$ and $i<\cf(\lambda)$ witnessing that $\alpha\in G(\vec f)$.
Find a sequence $\langle (\alpha_\tau,\beta_\tau)\mid \tau<\theta\rangle\in\prod_{\tau<\theta}A\times B$
such that $\sup_{\tau<\theta}\alpha_\tau=\alpha$ and $\alpha_\tau<\beta_\tau<\alpha_{\tau+1}$ for every $\tau<\theta$.
Next, for each $\tau<\theta$, fix a large enough $i_\tau<\cf(\lambda)$ such that:
\begin{itemize}
\item $i\le i_\tau$;
\item for every $j\in[i_\tau,\cf(\lambda))$, $f_{\alpha_\tau}(j)<f_{\beta_\tau}(j)<f_{\alpha_{\tau+1}}(j)$.
\end{itemize}
As in the previous case, we exploit the fact that $\theta\neq\cf(\lambda)$  to find some $\eta<\cf(\lambda)$ such that $i_{\tau}\le \eta$ for cofinally many $\tau<\theta$.
Let $B':=\{ \beta_\tau\mid \tau<\theta, i_{\tau}\le \eta\}$.
To see that $\beta\mapsto d(\eta,\beta)$ is injective over $B'$,
fix arbitrary $\tau<\tau'<\theta$ such that $\beta_\tau$ and $\beta_{\tau'}$ are in $B'$.
As $\eta\ge\max\{i,i_\tau,i_{\tau'}\}$, we infer that 
$$d(\eta,\beta_\tau)=f_{\beta_\tau}(\eta)<f_{\alpha_{\tau+1}}(\eta)\le f_{\alpha_{\tau'}}(\eta)<f_{\beta_{\tau'}}(\eta)=d(\eta,\beta_{\tau'}).$$

(4) 
Let $d:\aleph_0\times\aleph_{\omega+1}\rightarrow\aleph_\omega$ be arbitrary. 
Also, let $C$ be an arbitrary club in $\aleph_{\omega+1}$. 
Assuming that $(\aleph_{\omega+1}, \aleph_\omega) \twoheadrightarrow (\aleph_1, \aleph_0)$ holds,
we may now fix an elementary submodel $M\prec (\mathcal H_{\aleph_{\omega+2}},{\in})$
with $C,d\in M$
such that $\alpha:=\sup(M\cap\aleph_{\omega+1})$ has cofinality $\aleph_1$ and $M\cap\aleph_\omega$ is countable.
Clearly, $\alpha\in C$, $B:=M\cap\aleph_{\omega+1}$ is cofinal in $\alpha$ and, for all $\eta<\aleph_0$ and $\beta\in B$,
$d(\eta,\beta)$ is in the countable set $M\cap\aleph_\omega$.
So, $\alpha\in E^{\aleph_{\omega+1}}_{\aleph_1}\cap C\setminus A(d)$.
\end{proof}

The next theorem yields Theorem~C.

\begin{theorem}\label{thmc} Suppose that $\chi\le\theta<\lambda$ and $\theta\neq\cf(\lambda)$. Then the following are equivalent:
\begin{enumerate}
\item $\{ S\in I[\lambda^+]\mid S\s  E^{\lambda^+}_\theta\}\s J_\chi[\lambda^+]$;
\item $J_\chi[\lambda^+]$ contains a stationary subset of $E^{\lambda^+}_\theta$;
\item $m(\theta,\lambda,\theta,\chi)=\lambda$;
\item $m(\lambda,\lambda,\theta,\chi)=\lambda$.
\end{enumerate}
\end{theorem}
\begin{proof} 
$(1)\implies(2)$: By \cite[Lemma~4.4]{Sh:351} and \cite[\S1]{MR1261217}, $I[\lambda^+]$ contains some stationary subset of $E^{\lambda^+}_\theta$.

$(2)\implies(3)$: By the hypothesis, in particular, we may fix $\alpha\in E^{\lambda^+}_\theta\setminus(\lambda+1)$ and a sequence of functions 
 $\langle f_i:\alpha\rightarrow[\lambda]^{<\chi}\mid i<\alpha\rangle$ satisfying the following.
For every regressive function $f:\alpha\rightarrow\lambda$,
and every cofinal subset $B\s\alpha$, there exists some $i<\alpha$ such that $\sup\{ \beta\in B\mid f(\beta)\in f_i(\beta)\}=\alpha$.
Fix a strictly increasing function $\pi:\theta\rightarrow\alpha\setminus\lambda$ whose image is cofinal in $\alpha$.
We claim that $\mathcal H:=\{ f_i\circ \pi\mid i<\alpha\}$ is as sought.
To see this, let $g:X\rightarrow\lambda$ be an arbitrary function with $X\in[\theta]^\theta$.
Evidently, $B:=\pi[X]$ is cofinal subset of $\alpha$.
Now, fix any regressive function $f:\alpha\rightarrow\lambda$ extending $g\circ \pi^{-1}$.
Find $i<\alpha$ such that $B^*:=\{ \beta\in B\mid f(\beta)\in f_i(\beta)\}$ is cofinal in $\alpha$.
In particular, $X^*:=\{ \xi\in X\mid \pi(\xi)\in B^*\}$ has size $\theta$. Set $h:=f_i\circ \pi$, so that $h\in\mathcal H$.
Then, for every $\xi\in X^*$, 
$$g(\xi)=g(\pi^{-1}(\pi(\xi)))=f(\pi(\xi))\in f_i(\pi(\xi))=h(\xi).$$

$(3)\implies(4)$: By Lemma~\ref{lemma23b}(2), $m(\theta,\lambda)\le\lambda$.
Then by Lemma~\ref{lemma23b}(3), $m(\lambda,\lambda,\theta,\chi)\le\lambda$.
Altogether, $m(\lambda,\lambda,\theta,\chi)=\lambda$.

$(4)\implies(1)$: Fix arbitrary $S\in I[\lambda^+]$ with $S\s  E^{\lambda^+}_\theta$.
If $\lambda$ is regular, then by Proposition~\ref{sufficient}(1),
we may fix a map $d:\lambda\times\lambda^+\rightarrow\lambda$ such that $S\s A(d)$.
If $\lambda$ is singular, then fix a scale 
$\vec f=\langle f_\beta\mid \beta<\lambda^+\rangle$ for $\lambda$;
by \cite[Corollary~2.15]{MR2078366}, since $S\in I[\lambda^+]$,
there exists a club $C\s\lambda^+$ such that $C\cap S\s G(\vec f)$.
Then by Proposition~\ref{sufficient}(3), $C\cap S\s A(d)$.
So, in both cases, it follows from Lemma~\ref{lemma23} that $S\in J_\chi[\lambda^+]$.
\end{proof}

\section{The connection to Souslin trees}\label{seconnection}
Throughout this short section, $\kappa$ denotes a regular uncountable cardinal.
A \emph{$C$-sequence over $\kappa$} is a sequence $\vec C=\langle C_\alpha\mid\alpha<\kappa\rangle$
such that, for every $\alpha<\kappa$, $C_\alpha$ is a closed subset of $\alpha$ with $\sup(C_\alpha)=\sup(\alpha)$.
The sequence $\vec C$ is \emph{coherent} iff for every $\alpha<\kappa$ and every $\bar\alpha\in\acc(C_\alpha)$,
$C_{\bar\alpha}=C_\alpha\cap\bar\alpha$.
For a stationary subset $S\s\kappa$, the principle $\boxtimes^-(S)$ from \cite[Definition~1.3]{paper22} asserts the existence of a coherent $C$-sequence $\vec C=\langle C_\alpha\mid\alpha<\kappa\rangle$
such that, for every cofinal $A\s\kappa$, there exists a nonzero $\alpha\in S$ (equivalently, stationarily many $\alpha\in S$) with $\sup(\nacc(C_\alpha)\cap A)=\alpha$.

The following is an easy improvement of \cite[Theorem~4.3]{paper24}.

\begin{theorem}\label{m} Suppose that $\square(\kappa)$ holds and $\kappa^{<\kappa}=\kappa$.
For every stationary $S\in J_\omega[\kappa]$,
$\boxtimes^-(S)$ holds.
\end{theorem}
\begin{proof} Suppose $S\in J_\omega[\kappa]$ is stationary.
For simplicity, we shall assume that $S=S\cap\acc(\kappa)$.
By Propositions \ref{inaccessible} and \ref{prop35},
$\kappa$ is the successor of an uncountable cardinal.
So, by the main result of \cite{MR2596054},
since $\kappa^{<\kappa}=\kappa$, $\diamondsuit(\kappa)$ holds.
Using \cite[Proposition~3.2]{paper24},
we fix a matrix $\langle A^i_\gamma\mid i<\kappa,\gamma<\kappa\rangle$ 
such that for every sequence $\vec A=\langle A^i\mid i<\kappa\rangle$ of cofinal subsets of $\kappa$,
the following set is stationary: $$G(\vec{A}):=\{\gamma<\kappa\mid \forall i<\gamma(\sup(A^i\cap\gamma)=\gamma\ \&\ A^i\cap\gamma=A^i_\gamma)\}.$$

Next, using \cite[Proposition~3.5]{paper24}, we fix a coherent $C$-sequence $\langle C_\alpha\mid\alpha<\kappa\rangle$ such that, for every club $D\s\kappa$,
there exists $\alpha\in S$ with $\sup(\nacc(C_\alpha)\cap D)=\alpha$.

Let us also fix a club $C\s\kappa$ and a sequence  of functions  $\langle f_i:\kappa\rightarrow[\kappa]^{<\omega}\mid i<\kappa\rangle$ 
witnessing together that $S\in J_\omega[\kappa]$.
Next, for all $i,\alpha<\kappa$, let 
$$C_\alpha^i:=C_\alpha\cup\{\min(X\setminus(\sup(C_\alpha\cap\beta)+1))\mid\beta\in\nacc(C_\alpha), \beta>0, X\in\mathcal X_\beta^i\},$$
 where
$\mathcal X^i_\beta:=\{ \{\beta\}\cup A^i_{\gamma}\mid \gamma\in f_i(\beta)\}$.

A proof similar to that of \cite[Claim~4.3.1]{paper24} (see also \cite[Lemma~4.9]{paper23})
establishes that for all $i<\kappa$.
$\langle C^i_\alpha\mid \alpha<\kappa\rangle$ is a coherent $C$-sequence.

\begin{claim} There exists $i<\kappa$ such that, for every cofinal $A\s\kappa$,
there exists $\alpha\in S$ with $\sup(\nacc(C^i_\alpha)\cap A)=\alpha$.
\end{claim}
\begin{proof} Suppose not. Fix a sequence $\vec A=\langle A^i\mid i<\kappa\rangle$ of cofinal subsets of $\kappa$
such that for all $i<\kappa$ and $\alpha\in S$,
$\sup(\nacc(C^i_\alpha)\cap A^i)<\alpha$. Set $G:=G(\vec{A})$, so that $G$ is a stationary subset of $\kappa$
and $D:=\{\beta<\kappa\mid \sup(G\cap\beta)=\beta>0\}$ is a club in $\kappa$.
Pick $\alpha\in S$ such that $\sup(\nacc(C_\alpha)\cap C\cap D)=\alpha$. In particular, $\alpha\in C$.
Put $B:=\nacc(C_\alpha)\cap D$. For all $\beta\in B$, since $\beta\in D$, we know that the relative interval $G\cap(\sup(C_\alpha\cap\beta),\beta)$ is nonempty.
Consequently, we may find some regressive function $f:\alpha\rightarrow\alpha$ such that $f(\beta)\in G\cap(\sup(C_\alpha\cap\beta),\beta)$ for all $\beta\in B$.
Pick $i<\alpha$ and a cofinal subset $B'\s B$ such that $f(\beta)\in f_i(\beta)$ for all $\beta\in B'$.
Fix a large enough $\eta\in C_\alpha$ such that $\sup(C_\alpha\cap\eta)\ge i$.
By omitting an initial segment, we may assume that  $B'\cap\eta=\emptyset$.

Let $\beta\in B'$ be arbitrary. Write $\gamma:=f(\beta)$. Then $\gamma\in G\cap(\sup(C_\alpha\cap\beta),\beta)\s(i,\beta)$.
In particular, $\sup(A^i\cap\gamma)=\gamma$ and $A^i\cap\gamma=A^i_\gamma$, so that $X:=\{\beta\}\cup(A^i\cap\gamma)$
is in $\mathcal X^i_\beta$,
and $\min(X\setminus(\sup(C_\alpha\cap\beta)+1))\in A^i\cap\gamma$.
It follows that for all $\beta\in B'$, $C^i_\alpha\cap A^i\cap(\sup(C_\alpha\cap\beta),\beta)$ is a finite nonempty set,
contradicting the fact that $\sup(\nacc(C^i_\alpha)\cap A^i)<\alpha=\sup(B')$.
\end{proof}
Let $i<\kappa$ be given by the claim.
Then $\langle C_\alpha^i\mid \alpha<\kappa\rangle$ is a $\boxtimes^-(S)$-sequence.
\end{proof}

Following \cite[\S2]{paper48}, a $\lambda^+$-Souslin tree is said to be \emph{maximally-complete} iff it is $\chi$-complete
for the least cardinal $\chi$ to satisfy $\lambda^\chi>\lambda$.
Another concept is that of a \emph{club-regressive} tree from \cite[p.~1961]{paper22};
such a tree does not have a copy of the Cantor tree,
and hence it is nowhere $\aleph_1$-complete.
For an application of the two extremes, see \cite[Corollary~2.23]{paper48}.

\begin{lemma}\label{lemma42} Suppose that $\lambda$ is an uncountable cardinal and $2^\lambda=\lambda^+$.
\begin{enumerate}
\item If $\boxtimes^-(\lambda^+)$ holds, then there exist a club-regressive $\lambda^+$-Souslin tree
and a maximally-complete  $\lambda^+$-Souslin tree;
\item If $T\s\lambda^+$ is a stationary set all of whose stationary subsets reflect,
then $\boxtimes^-(T)$ entails the existence of a uniformly coherent $\lambda^+$-Souslin tree.
\end{enumerate}
\end{lemma}
\begin{proof}  By the main result of \cite{MR2596054},
since $2^{\lambda}=\lambda^+$, $\diamondsuit(\lambda^+)$ holds.

(1) By \cite[Proposition~2.3]{paper22}, $\boxtimes^-(\lambda^+)$ together with $\diamondsuit(\lambda^+)$ entails the existence of a club-regressive $\lambda^+$-Souslin tree.
In addition, for every cardinal $\chi$ such that $\lambda^{<\chi}=\lambda$,
by \cite[Theorem~4.13]{paper24}, $\boxtimes^-(\lambda^+)$  implies $\boxtimes'(E^{\lambda^+}_{\ge\chi})$,
and then \cite[Proposition~4.11]{paper24} entails the existence of a $\chi$-complete $\lambda^+$-Souslin tree.

(2) By \cite[Theorem~3.11(2)]{paper28} together with \cite[Theorem~6.35]{paper23}.
\end{proof}

\section{Applications}
\begin{cor}\label{cor31} 
Suppose that $\lambda$ is an uncountable cardinal.
Let $\chi\in\{2,\omega\}$. Then
$J_\chi[\lambda^+]$ is the $\lambda^+$-complete ideal generated by $$\{ A(d)\cap E^{\lambda^+}_\theta\mid \theta<\lambda=m(\theta,\lambda,\theta,\chi), d\in{}^{\lambda^+\times\lambda^+}\lambda\}\cup \ns_{\lambda^+}.$$
\end{cor}
\begin{proof} 

By Definition~\ref{regressive}, $J_\chi[\lambda^+]$ is a $\lambda^+$-complete ideal covering $\ns_{\lambda^+}$.
Now, given a cardinal $\theta<\lambda$ such that $m(\theta,\lambda,\theta,\chi)=\lambda$ and a map $d\in{}^{\lambda^+\times\lambda^+}\lambda$,
to see that $A(d)\cap E^{\lambda^+}_\theta$ is in $J_\chi[\lambda^+]$, suppose that the former is nonempty, so that $\theta$ is regular.
By Proposition~\ref{prop22}(1) and Lemma~\ref{lemma23b}(2), $\cf(\lambda)\neq\theta$, and then
the implication $(3)\implies(4)$ of Theorem~\ref{thmc} implies that  $m(\lambda,\lambda,\theta,\chi)=\lambda$.
Then, by Lemma~\ref{lemma23}, $A(d)\cap E^{\lambda^+}_\theta\in J_\chi[\lambda^+]$.

Next, suppose that $S\in J_\chi[\lambda^+]$. As we are considering a $\lambda^+$-complete generated ideal,
it suffices to prove that for each cardinal $\theta\le\lambda$, 
if $(S\cap E^{\lambda^+}_\theta)\notin\ns_{\lambda^+}$
then $m(\theta,\lambda,\theta,\chi)=\lambda$ and there exists a map $d\in{}^{\lambda^+\times\lambda^+}\lambda$
such that $(S\cap E^{\lambda^+}_\theta\setminus A(d))\in \ns_{\lambda^+}$.
To this end, let $\theta\le\lambda$ be such that $S\cap E^{\lambda^+}_\theta$ is stationary.
In particular, $\theta$ is regular,
and, by Proposition~\ref{prop35}, $\theta\neq\cf(\lambda)$.
Now, by the implication $(2)\implies(3)$ of Theorem~\ref{thmc}, $m(\theta,\lambda,\theta,\chi)=\lambda$.
Finally, let $C$ and $\langle f_i\mid i<\lambda^+\rangle$ witness together that $S\in J_\chi[\lambda^+]$.
Without loss of generality, $\min(C)>\theta$.
Define a map $d:\lambda^+\times\lambda^+\rightarrow\theta$ via $d(\eta,\beta)=\sup(f_\eta(\beta)\cap\theta)$.
Now, given $\alpha\in C\cap S\cap E^{\lambda^+}_\theta$ and a cofinal subset $B\s\alpha$, 
we need to find $\eta<\alpha$ and a cofinal $B'\s B$ on 
which $\beta\mapsto d(\eta,\beta)$ is injective. 

Let $\bar B$ be a cofinal subset of $B$ of order-type $\theta$.
Pick a regressive function $f:\alpha\rightarrow\alpha$ such that $f(\beta)=\otp(\bar B\cap\beta)$ for all $\beta\in\bar B$.
Now, find $i<\alpha$ such that $B^*:=\{\beta\in\bar B\mid f(\beta)\in f_i(\beta)\}$ is cofinal in $\alpha$.
If there exists a cofinal $B'\s B^*$  on 
which $\beta\mapsto d(i,\beta)$ is injective, then we are done.
Otherwise, since $\otp(B')=\theta$, 
we may find a cofinal $B'\s B^*$ along with some $\tau<\theta$
such that $d(i,\beta)=\tau$ for all $\beta\in B'$.
Pick a large enough $\beta\in B'$ such that $\otp(\bar B\cap\beta)>\tau$.
As $\beta\in B^*$, $\tau<\otp(\bar B\cap\beta)=f(\beta)\le\sup(f_i(\beta)\cap\theta)=d(i,\beta)$. 
This is a contradiction.
\end{proof}

At the level of $\aleph_2$, the preceding characterization can be simplified. The next two results yield together Theorem~B:
\begin{cor}\label{thmb} The following are equivalent:
\begin{enumerate}
\item $J[\aleph_2]=\{ S\s\aleph_2\mid S\cap S^2_1\text{ is nonstationary}\}$;
\item $J[\aleph_2]\neq\ns_{\aleph_2}$;
\item $m(\aleph_0,\aleph_0,\aleph_0)=\aleph_1$;
\item There exists a nonmeager set of reals of size $\aleph_1$.
\end{enumerate}
\end{cor}
\begin{proof} 
$(1)\implies(2)$:  $S^2_0$ is a stationary subset of $\aleph_2$ disjoint from $S^2_1$.

$(2)\implies(3)$: As $J[\aleph_2]\neq\ns_{\aleph_2}$, it follows from Corollary~\ref{cor31} and monotonicity
that $m(\aleph_0,\aleph_0,\aleph_0)\le m(\aleph_0,\aleph_1,\aleph_0)=\aleph_1$. 
Also, by a diagonalization argument, $m(\aleph_0,\aleph_0,\aleph_0)\ge\aleph_1$.

$(3)\iff(4)$: By Proposition~\ref{bprop}.
\end{proof}

\begin{cor}  The following are equivalent:
\begin{enumerate}
\item $J_\omega[\aleph_2]=\{ S\s\aleph_2\mid S\cap S^2_1\text{ is nonstationary}\}$;
\item $J_\omega[\aleph_2]\neq\ns_{\aleph_2}$;
\item $m(\aleph_0,\aleph_0,\aleph_0,\aleph_0)=\aleph_1$;
\item $\mathfrak b=\aleph_1$.
\end{enumerate}
\end{cor}
\begin{proof} The proof is similar to that of Corollary~\ref{thmb}.
\end{proof}

By Lemma~\ref{lemma23b}(1), the next corollary improves \cite[Corollary~2.5]{paper24}:
\begin{cor}\label{lemma26}  For cardinals $\theta=\cf(\theta)<\lambda$,
if 
$\max\{m(\theta,\theta,\theta),m(\theta,\lambda)\}=\lambda$
then 
$J[\lambda^+]$ contains a stationary subset of $E^{\lambda^+}_\theta$.
\end{cor}
\begin{proof} By Lemma~\ref{lemma23b}(3), the hypothesis implies that $m(\lambda,\lambda,\theta)=\lambda$.
Now, appeal to the implication $(4)\implies(2)$ of Theorem~\ref{thmc}.
\end{proof}

\begin{cor}\label{cor34} For cardinals $\chi\le\theta\le\lambda$ such that 
$m(\theta,\theta,\theta,\chi)\le\lambda<\theta^{+\theta}$, $J_\chi[\lambda^+]$ contains a stationary subset of $E^{\lambda^+}_\theta$.
\end{cor}
\begin{proof} To avoid trivialities, suppose that $E^{\lambda^+}_\theta\neq\emptyset$, so that $\theta$ is regular.
Then, by Proposition~\ref{prop22}(3), $m(\theta,\lambda)=\lambda$.
So, by Lemma~\ref{lemma23b}(3), $m(\lambda,\lambda,\theta,\chi)=\lambda$,
Now, appeal to the implication $(4)\implies(2)$ of Theorem~\ref{thmc}.
\end{proof}

\begin{cor} If $2^{\aleph_0}<\aleph_\omega$, then for co-finitely many $n<\omega$, $S^{n}_0\in J[\aleph_{n}]$.\qed
\end{cor}

In the following, $\ssh$ stands for \emph{Shelah's Strong Hypothesis} \cite[\S8.1]{MR1112424}.

\begin{cor}\label{cor33} Suppose that $\lambda$ is an uncountable cardinal for which $J[\lambda^+]=\ns_{\lambda^+}$. Then:
\begin{enumerate}
\item $\lambda<\beth_{\omega}$.
If $\ssh$ holds, then $\lambda\le 2^{\aleph_1}$;
\item for every regular cardinal $\theta<\lambda$,
either $2^\theta>\lambda$ or $\theta^{+\theta}\le\lambda$.
\end{enumerate}
\end{cor}
\begin{proof} (1) The first part is explained in the proof of \cite[Corollary~4.7]{paper24}.
Now, given $\lambda>2^{\aleph_1}$, find $\theta\in\{\aleph_0,\aleph_1\}$ such that $\cf(\lambda)\neq\theta$.
By \cite[Corollary~2.19]{MR3356615}, since $\ssh$ holds, it follows from $\lambda>2^\theta$ and $\cf(\lambda)\neq\theta$ that
$\cf([\lambda]^\theta,{\supseteq})=\lambda$. So, by Lemma~\ref{lemma23b}(1), $m(\theta,\lambda)=\lambda$.
In addition, $\lambda>2^\theta\ge m(\theta,\theta,\theta)$.
Now, appeal to Corollary~\ref{lemma26}.

(2) By Corollary~\ref{cor34}.
\end{proof}

\begin{cor}\label{cor36} Suppose that
$\lambda$ is an uncountable cardinal such that $2^\lambda=\lambda^+$, $\square(\lambda^+)$ holds,
and there exists a  regular cardinal $\theta$ such that $m(\theta,\theta,\theta,\omega)\le\lambda<\theta^{+\theta}$ (e.g., $2^\theta\le\lambda<\theta^{+\theta}$).
Then there exist a club-regressive $\lambda^+$-Souslin tree
and a maximally-complete $\lambda^+$-Souslin tree.

If, in addition, all stationary subsets of $E^{\lambda^+}_{\theta}$ reflect,
then there exists a uniformly coherent $\lambda^+$-Souslin tree.
\end{cor}
\begin{proof} 
Suppose that there exists a cardinal $\theta$ as above.
Then, by Corollary~\ref{cor34},  $J_\omega[\lambda^+]$ contains some stationary subset $S\s E^{\lambda^+}_\theta$.
Then, by Theorem~\ref{m}, $\boxtimes^-(T)$ holds for every stationary subset $T\s S$.
Now, appeal to Lemma~\ref{lemma42}.
\end{proof}

\begin{prop}\label{prop39} Suppose that $\theta=\theta^{<\theta}<\mu$ are cardinals.
In the forcing extension by $\add(\theta,\mu)$, $m(\theta,\theta,\theta)=\theta^+$.
\end{prop}
\begin{proof} By a diagonalization argument, $m(\theta,\theta,\theta)\ge\theta^+$,
so we focus on the verifying the other inequality.
Let $\mathbb P$ denote $\add(\theta,\mu)$.
Recall that the conditions in $\mathbb P$ are functions $p:a\rightarrow\theta$ with $a\s \theta\times \mu$ and $|a|<\theta$.
This is a $\theta$-closed notion of forcing, and since $\theta^{<\theta}=\theta$, it has the $\theta^{+}$-cc,
so that $\mathbb P$ preserves the cardinal structure.
Let $H$ be $V$-generic, and work in $V[H]$. Let $h:=\bigcup H$, so that $h$ is a function from $\theta\times\mu$ to $\theta$.
For every $\beta<\theta^+$, derive the fiber map $h_\beta:\theta\rightarrow\theta$ via $h_\beta(\alpha):=h(\alpha,\beta)$.
To show that $m(\theta,\theta,\theta)\le\theta^+$,
it suffices to prove that for every function $g:X\rightarrow\theta$ with $X\in[\theta]^\theta$,
there exists $\beta<\theta^+$ such that $|g\cap h_\beta|=\theta$.
Now, if this fails, we may fix $g$ as above along with $\tau<\theta$ and $B\in[\theta^+]^{\theta^+}$ such that
$\dom(g\cap h_\beta)\s \tau$ for every $\beta\in B$,
and then, setting $A:=X\setminus\tau$,
we would get that $g(\alpha)\notin\{ h_\beta(\alpha)\mid \beta\in B\}$ for every $\alpha\in A$.
Therefore, it suffices to prove the following.
\begin{claim} In $V[H]$, for all $A\in[\theta]^{\theta}$ and $B\in[\theta^+]^{\theta^+}$,
and every function $g:A\rightarrow\theta$, there exist $\alpha\in A$ and $\beta\in B$ such that $g(\alpha)=h(\alpha,\beta)$.
\end{claim}
\begin{proof} We run a standard density argument (see, for instance, the proof of \cite[Theorem~27]{strongcoloringpaper}). 
Let $p$ be an arbitrary condition in $\mathbb P$
forcing that $\dot A,\dot B,\dot g$ are names for objects as above.
As $\mathbb P$ has the $\theta^+$-cc, we may assume that $|\dot A|=|\dot g|=\theta$.
Fix an elementary submodel $M\prec (\mathcal H_{\theta^{++}},{\in})$ of size $\theta$ such that ${}^{<\theta}M\s M$ and $\mathbb P,\dot A,\dot B,\dot g\in M$.
Find an extension $p'$ of $p$ and an ordinal $\beta\in\theta^+\setminus M$ such that $p'\Vdash\check\beta\in\dot B$.
Set $q:=p'\cap M$ and note that $q$ is a legitimate condition lying in $M$.
Fix a large enough $\epsilon<\theta$ such that 
$\{ \alpha<\theta\mid \exists \eta<\mu[(\alpha,\eta)\in\dom(p')]\}\s\epsilon$.
Now, in $M$, find a condition $q'$ extending $q$, an ordinal $\alpha\in\theta\setminus\epsilon$ and an ordinal $\tau<\theta$
such that $q'\Vdash\check\alpha\in\dot A$ and $q'\Vdash \dot g(\check \alpha)=\check \tau$.
As $q'$ is an extension of $q$ lying in $M$, the definition of $q$ entails that $p'\cup q'$ is a legitimate condition.
As $\alpha\notin\epsilon$ and $\beta\notin M$, it follows that $(\alpha,\beta)\notin\dom(p'\cup q')$.
So, $p'\cup q'\cup\{ ((\alpha,\beta),\tau)\}$ is a legitimate condition forcing that 
$g(\alpha)=h(\alpha,\beta)$, as desired.
\end{proof}
This completes the proof.
\end{proof}

\begin{cor}\label{cor310} Suppose that $\theta,\lambda$ are cardinals such that:
\begin{itemize}
\item $\theta<\lambda<\theta^{+\theta}$;
\item $\theta^{<\theta}=\theta$ and $2^\lambda=\lambda^+$;
\item $\square(\lambda^+)$ holds. 
\end{itemize}
After forcing with $\add(\theta,\theta^+)$,
there exist a club-regressive $\lambda^+$-Souslin tree
and a maximally-complete $\lambda^+$-Souslin tree.

If, in the extension, 
every stationary subset of $E^{\lambda^+}_{\theta}$ reflect,
then there exists a uniformly coherent $\lambda^+$-Souslin tree.
\end{cor}
\begin{proof} By Proposition~\ref{prop39} and Corollary~\ref{cor36}.
\end{proof}

\begin{cor}  Suppose that $\theta=\theta^{<\theta}<\mu$ are regular cardinals such that $\mu$ is not weakly compact in $L$.
Then, in any forcing extension by $\add(\theta,\mu)$ followed by a $\theta^+$-distributive forcing,
if $2^{\theta^+}=\theta^{++}=\mu$,
then there exists a $\mu$-Souslin tree.
\end{cor}
\begin{proof} By \cite[Claim~1.10]{MR908147}, the fact that $\mu$ is not weakly compact in $L$ implies that $\square(\mu)$ holds in any universe in which $\mu$ remains regular and uncountable.
By Corollary~\ref{prop39}, $m(\theta,\theta,\theta)=\theta^+$ in the forcing extension by $\add(\theta,\mu)$,
and this surely remains the case after a $\theta^+$-distributive forcing.
So, in the final model, Corollary~\ref{cor36} (using $\lambda:=\theta^+$)
implies that there exists a $\mu$-Souslin tree.
\end{proof}

The next result implies Theorem~A:

\begin{cor} Suppose that $n<\omega$, $\mathfrak b<2^{\aleph_{n}}=\aleph_{n+1}$ and $\square(\aleph_{n+1})$ holds.
Then:
\begin{enumerate}
\item there exists a club-regressive $\aleph_{n+1}$-Souslin tree;
\item there exists a maximally-complete $\aleph_{n+1}$-Souslin tree;
\item if there exists a stationary subset of $S^{n+1}_0$ all of whose stationary subsets reflect,
then there exists a uniformly coherent $\aleph_{n+1}$-Souslin tree.
\end{enumerate}
\end{cor}
\begin{proof} 
Set $\lambda:=\aleph_n$.
By Proposition~\ref{bprop}, $m(\aleph_0,\aleph_0,\aleph_0,\aleph_0)=\mathfrak b\le\lambda$.
So, by Corollary~\ref{cor34},  $J_\omega[\lambda^+]$ contains some stationary subset of $E^{\lambda^+}_\omega$.
As $\lambda$ is a regular uncountable cardinal, by \cite[Lemma~4.4]{Sh:351}, $E^{\lambda^+}_\omega\in I[\lambda^+]$,
so, by the implication $(2)\implies(1)$ of Theorem~\ref{thmc}, $E^{\lambda^+}_\omega\in J_\omega[\lambda^+]$.
Then, by Theorem~\ref{m}, $\boxtimes^-(T)$ holds for every stationary subset $T\s E^{\lambda^+}_\omega$.
Now, appeal to Lemma~\ref{lemma42}.
\end{proof}

\section{Some open problems}
Tanmay Inamdar pointed out that there are ground models in which adding a single Cohen real makes $\mathfrak b=\aleph_1$;
for instance, by \cite[Corollary~2.2.9 and Theorem~3.3.23]{MR1350295},
this is the case for a ground model in which $\cov(\mathcal M)=\aleph_1$.
Recalling Theorem~A, this raises the following question:
\begin{q} Suppose that $\square(\aleph_2)$ holds. 
Is it the case that after forcing to a add a single Cohen real, there exists an $\aleph_2$-Souslin tree?
\end{q}

By Corollary~\ref{cor33}(1) and the results of Section~\ref{seconnection},
for a tail of cardinals $\lambda$, if $2^{\lambda}=\lambda^+$ and $\square(\lambda^+)$ holds,
then there exists a $\lambda^+$-Souslin tree. We do not know whether this may be improved to cover all uncountable cardinals $\lambda$.
So, we ask:

\begin{q} Is it consistent that $2^{\aleph_1}=\aleph_2$, $\square(\aleph_2)$ holds,
but there are no $\aleph_2$-Souslin trees?
\end{q}

If the answer to the preceding is affirmative, then Corollary~\ref{cor310} gives a result along the lines of those from \cite{MR768264,paper6,MR3961614} concerning 
the existence of a certain $\kappa$-Aronszajn tree in an extension by a notion of forcing of size $<\kappa$.
\begin{q} Is it consistent that there exists a singular cardinal $\lambda$ 
such that, for every function $d:\lambda^+\times\lambda^+\rightarrow\lambda$,
$E^{\lambda^+}_{\neq\cf(\lambda)}\setminus A(d)$ is stationary?
\end{q}

\section*{Acknowledgments}
The author was partially supported by the Israel Science Foundation (grant agreement 2066/18) and by the European Research Council (grant agreement ERC-2018-StG 802756).

I thank {\v S}{\'a}rka Stejskalov{\'a} for an insightful feedback on  a preliminary version of this paper.

\end{document}